\documentclass[a4paper,10pt]{article}
\usepackage[utf8]{inputenc}
\usepackage{amsmath,amssymb,amsthm}
\usepackage[colorlinks,citecolor=blue]{hyperref}
\usepackage{cleveref}
\usepackage{enumitem}
\usepackage{mathrsfs}
\usepackage{graphicx}
\usepackage{tikz}
\usetikzlibrary{patterns}

\newcommand{\ZZ}{\mathbb{Z}}
\newcommand{\NN}{\mathbb{N}}

\newcommand{\QQ}{\mathbb{Q}}

\newtheorem{theorem}{Theorem}[section] 
\newtheorem{proposition}[theorem]{Proposition} 
 
\newtheorem{corollary}[theorem]{Corollary} 
\newtheorem{lemma}[theorem]{Lemma} 
\newtheorem{definition}{Definition}
 
\newtheorem{remark}[theorem]{Remark} 
\newtheorem{example}[theorem]{Example}

\newcommand{\Ch}{\operatorname{Ch}}
\newcommand{\Int}{\operatorname{Int}}
\newcommand{\Inc}{\operatorname{Inc}}
\newcommand{\rk}{\operatorname{rk}}

\newcommand{\barP}{\overline{P}}
\newcommand{\bott}{\hat{0}}
\newcommand{\topp}{\hat{1}}

\newcommand{\qZ}{\mathsf{Z}}
\newcommand{\qE}{\mathsf{E}}
\newcommand{\qL}{\mathsf{L}}
\newcommand{\ehr}{\operatorname{\textsf{Ehr}}}

\newcommand{\qserieZ}{\mathscr{Z}}

\newcommand{\qserieL}{\mathscr{L}}

\newcommand{\qbinom}[2]{\genfrac{[}{]}{0mm}{1}{#1}{#2}_q}

\newcommand{\HH}{\mathbb{H}}  
\newcommand{\rel}{\stackrel{\heartsuit}{\leftrightarrow}}  

\title{On a q-analogue of the Zeta polynomial of posets}
\author{F. Chapoton}
\date{\today}

\begin{document}

\maketitle

\begin{abstract}
  We introduce a $q$-analogue of the classical Zeta polynomial of
  finite partially ordered sets, as a polynomial in one variable $x$
  with coefficients depending on the indeterminate $q$. We prove some
  properties of this polynomial invariant, including its behaviour
  with respect to duality, product and disjoint union. The leading
  term is a $q$-analogue of the number of maximal chains, but not
  always with non-negative coefficients. The value at $q=0$ turns out
  to be essentially the characteristic polynomial.
\end{abstract}

In the study of finite partially ordered sets (posets), one uses frequently
polynomial invariants. They are useful to distinguish the posets
or recognize them under different disguises, but also for testing the
solidity of our understanding. Computing these invariants may require
a good handle on the combinatorial mechanism behind the scene. This
can lead to structural results. For example, sometimes these
polynomials factor nicely and one would like to understand the reason
for this.

One of the most classical such polynomial is the Zeta polynomial,
whose values at positive integers count chains of elements, introduced
in \cite[\S 3]{stanley1974}. This polynomial is related quite closely
to the order polynomial, introduced in \cite{stanley1970}, which
describes integer points in a polytope naturally attached to the
poset. Yet another polynomial is the characteristic polynomial,
recording the values of the Möbius function, introduced in
\cite{rota1964}. An important but less well-known polynomial is the
Coxeter polynomial, which contains information about the derived
category of modules over the incidence algebra. For more information
on this last topic, see the survey article \cite{delapena1994}





\smallskip

The aim of the present article is to introduce and study a
$q$-analogue of the Zeta polynomial, defined for posets endowed with
a height function. This is a polynomial in the variable $x$, whose
values at $q$-integers count chains according to the sum of heights of
their elements. We will show that it shares many properties of the
Zeta polynomial, to which it reduces when setting $q=1$.

\smallskip

The main motivation for this definition comes from the $q$-analogue of
Ehrhart theory introduced in \cite{q-ehrhart}, where the objects of
study are lattice polytopes together with a linear form. It is
well-known that the usual Zeta polynomial of the distributive lattice
$J(P)$ of lower ideals in a poset $P$ is equal to the order polynomial
of the partial order $P$, and therefore to the Ehrhart polynomial of
the order polytope of $P$ introduced in
\cite{stanley_two_polytopes}. The present article started with the aim
to understand what happens to this story when one replaces the Ehrhart
polynomial of $P$ by the $q$-Ehrhart polynomial. The conclusion is
that everything works nicely, and a rather clean theory can be
established.


The construction is mostly as expected, with the appropriate changes
and dependencies on the choice of the height function on the
poset. There is one unexpected property, namely the value of the
$q$-Zeta polynomial at $q=0$ can under some conditions be identified
with the characteristic polynomial.



\smallskip

The reader should be warned that, in some sense, a large part of the
surrounding context and proofs are not new and indeed very classical
since the pioneering works of Stanley on $P$-partitions. The precise
relation of $q$-Ehrhart polynomials with the theory of $P$-partitions
has been described in \cite[\S 4]{q-ehrhart}. A nice reference on the
history of the subject of $P$-partitions can be found in
\cite{gessel-P}, in which the connection between $P$-partitions and
Zeta polynomials appears in \S 8.1.

The study of Ehrhart generating series is the main focus of the theory
of $P$-partitions. From this perspective, the $q$-Zeta polynomials are
just describing in another compact way the coefficients of these
Ehrhart series. This is a slightly different viewpoint, which offers
other insights.

\smallskip

The article is organized as follows. After recalling briefly the
classical Zeta polynomial, the $q$-Zeta polynomial is defined and some
examples are given in Section \ref{sect1}. In Section \ref{sect2},
some basic properties are proved. In Section \ref{sect3}, the
relationship, in the case of distributive lattices, with the $q$-order
polynomial is spelled out. In Section \ref{sect4}, a property relating
the leading coefficient of the $q$-Zeta polynomial and the numerator
of the Ehrhart series is proved. In Section \ref{sect5}, a known
criterion is recalled for the positivity of the coefficients in the
numerator of the Ehrhart series. Section \ref{sect6} describes the
unexpected relationship, at $q=0$, with the characteristic
polynomial. Section \ref{sect7} proves a $q$-analogue of a reciprocity
theorem for Zeta polynomials of Eulerian posets. The article ends with
three appendices : \ref{app:flags} on classical results on flag
vectors, positivity and $R$-labellings, \ref{app:q-ring} on
$q$-analogues of polynomials with integer values and
\ref{app:q-incidence} on a $q$-analogue of the incidence algebra.

\smallskip

I would like to humbly dedicate this article to Richard Stanley, with my sincere admiration.

Acknowledgments: Many thanks to the careful referee that suggested to add the material of Section \ref{sect7}.




\smallskip

The author has made use of SageMath \cite{sage} in the preparation of this article.

The author has been supported by the ANR Combiné (ANR-19-CE48-0011) and ANR Charms (ANR-19-CE40-0017).

\section*{Introduction}

Let us first recall briefly the classical Zeta polynomial of posets.

Let $P$ be a finite poset. For every integer $n \geq 2$, consider the
set $C_n(P)$ of chains $e_1 \leq e_2 \leq \cdots \leq e_{n-1}$ in
$P$. Then the cardinality of $C_n(P)$ is given by the value at $n$ of
a polynomial $Z_P(x)$, called the Zeta polynomial of $P$.

One can easily prove this fact by the following computation.

For $k\geq 1$, let $\Ch_k(P)$ be the set of strict chains with $k$
elements in $P$, \textit{i.e.} sequences $c_1 < \cdots < c_k$.

By gathering chains according to their underlying strict chain and multiplicities, one finds that
\begin{equation}
  \label{classical_Z1}
  Z_P(n) = \sum_{e_1 \leq \cdots \leq e_{n-1}} 1
        = \sum_{k \geq 1} \left( \sum_{\stackrel{m_1,\ldots,m_k \geq 1}{\sum_i m_i=n-1}} 1 \right) \# \Ch_k.
\end{equation}
But the inner sum is just the binomial coefficient $\binom{n-2}{k-1}$, so that
\begin{equation}
  \label{classical_Z2}
  Z_P(n) = \sum_{k \geq 1} \binom{n-2}{k-1} \# \Ch_k,
\end{equation}
which is obviously a polynomial evaluated at $n$.


The Zeta polynomial is a useful invariant of posets, with nice general
properties. By Formula \eqref{classical_Z2}, it belongs to the ring of
polynomials with rational coefficients taking integer values on
$\NN$. The Zeta polynomial is multiplicative with respect to the
Cartesian product of posets, additive with respect to disjoint union
and invariant under duality. For more on this classical subject, see
for instance \cite[\S 3.11]{stanleyEC1}.

\section{$q$-analogue}

\label{sect1}

Let us now turn to our proposal for a $q$-analogue of the Zeta polynomial.

The letter $q$ stands for the indeterminate in $\QQ(q)$. We use the
standard notation $[n]_q = (q^n-1)/(q-1)$ for the $q$-analogue of the integer $n \in \ZZ$, which is a polynomial in $q$ if $n \geq 0$ and a Laurent polynomial otherwise. We will denote
by $[n]!_q$ the $q$-factorial of $n$ when $n \geq 0$. Note that
$[0]_q=0$ and $[1]_q=1$.

\smallskip

Let $P$ be a finite poset. We need the additional data of a height function
$h : P \to \NN$ such that $h(x) < h(y)$ for every cover relation
$x < y$ in $P$.

A poset $P$ is \textit{graded} if there exists a height function that
increase by $1$ along every cover relation. Then there is a preferred
choice for such a height function $h$, by assuming further that it
has minimal value $0$ on every connected component of $P$. This
specific height function will be denoted by $\rk$.

Every poset $P$ can be endowed with a height function by choosing an
arbitrary linear extension and using it as a height function.

Note that some posets that are not graded are nevertheless naturally
endowed with natural height functions, for example the Tamari lattices
using their description as posets of tilting modules \cite{thomas_tamari}.

\medskip

For a strict chain $c=(c_1,\ldots,c_{k})$ in $\Ch_k(P)$, let $h(c)$
denote the sequence of heights $(h(c_1),\ldots,h(c_{k}))$ and let
$\sum h(c)$ denote the sum of this sequence. The hypothesis on $h$
ensures that the sequence $h(c)$ is strictly increasing.

For every integer $n \geq 2$, let us define
\begin{equation}
  \label{qZ_sum_over_chains}
  \qZ_{P,h}([n]_q) = \sum_{e_1 \leq \dots \leq e_{n-1}} q^{\sum_j h(e_j)}
        = \sum_{k \geq 1} \sum_{c \in \Ch_k(P)} \sum_{\stackrel{m_1,\ldots,m_k\geq 1}{\sum_i m_i=n-1}} q^{\sum_i m_i h(c_i)}.
\end{equation}
For the moment, the left hand side is an abuse of notation, as we will
only prove later that the right hand side is indeed the evaluation of
a polynomial at $[n]_q$. The reader can safely but temporarily assume
that the left-hand side just means some function of $n$ and $q$
depending of $P$ and $h$.

This can be rewritten as
\begin{equation}
  \label{qZ_step}
  \qZ_{P,h}([n]_q)
        = \sum_{k \geq 1} \sum_{c \in \Ch_k(P)} q^{\sum h(c)}\sum_{\stackrel{m'_1,\ldots,m'_k \geq 0}{\sum_i m'_i=n-k-1}} q^{\sum_i m'_i h(c_i)}.
\end{equation}
At this point, one can recognize the innermost sum as being
essentially the $q$-Ehrhart polynomial of a simplex. Before proceeding
further, let us recall the theory of $q$-Ehrhart polynomials, whose
details can be found in \cite{q-ehrhart}.

\subsection{$q$-Ehrhart polynomials and reciprocity}

Fix a lattice polytope $Q$ and an integral linear form $\ell$ on the ambient
lattice. Assume that $\ell$ is not constant on any edge
of $Q$ and takes values in $\NN$ on $Q$. Then there exists a unique
polynomial $\ehr_{Q,\ell}(x) \in \QQ(q)[x]$ such that
\begin{equation}
  \ehr_{Q,\ell}([n]_q) = \sum_{z \in n Q} q^{\ell(z)},
\end{equation}
for every integer $n \geq 0$.  In words, this is counting lattice
points in the dilates of $Q$ according to the value of $\ell$ on each
lattice point. By the formula above, the polynomial $\ehr_{Q,\ell}$
belongs to the ring of polynomials whose value at every $q$-integer
$[n]_q$ with $n \geq 0$ is a polynomial in $q$ with positive integer
coefficients. This is called the $q$-Ehrhart polynomial of the lattice
polytope $Q$ with respect to the linear form $\ell$. Note that setting
$q=1$ recovers the classical Ehrhart polynomial of lattice polytopes,
not depending on the linear form $\ell$. The degree of $\ehr_{Q,\ell}$
is the maximum value of $\ell$ on $Q$.

Let us also recall Ehrhart reciprocity in this setting. Let $n \geq 1$ be an
integer. Then the evaluation of $\ehr_{Q,\ell}$ at $[-n]_q$ is given
up to sign by the similar sum over interior points in the dilates of $Q$:
\begin{equation}
  \ehr_{Q,\ell}([-n]_q) = (-1)^d \sum_{z \in \Int(n Q)} q^{-\ell(z)},
\end{equation}
where $d$ is the dimension of the polytope $Q$ and $\Int$ denotes the
interior of a polytope. Note that the interior of a $0$-dimensional
polytope (a point) is just itself.

\medskip

For $k\geq 1$, the lattice polytope in $\NN^k$ whose vertices are the
basis vectors in $\NN^k$ will be called the standard basic simplex.

For every $k$-tuple $a$ of distinct elements of $\NN$, let us denote
by $\qE_{a}$ the $q$-Ehrhart polynomial of the standard basic simplex
in $\NN^k$ with respect to the linear form given by the standard
scalar product with $a$. The degree of $\qE_{a}$ is the maximal
element of $a$. At $q=1$, $\qE_{a}$ becomes the Ehrhart polynomial of
the standard basic simplex, namely $\binom{x+k-1}{k-1}$, no longer
depending on $a$.

For example, for the tuple $a=(1,2,3)$, one gets
\begin{equation*}
  \qE_a = (\left(q - 1\right) x + 1) \cdot (q x + 1) \cdot \frac{(q^{2} x + q + 1)}{q + 1}.
\end{equation*}

\subsection{Definition of $q$-Zeta polynomial}

Let us go back to our proposed definition \eqref{qZ_step} for $\qZ_{P,h}$. One therefore finds that
\begin{equation}
  \label{qZ_almost}
  \qZ_{P,h}([n]_q)
        = \sum_{k \geq 1} \sum_{c \in \Ch_k(P)} q^{\sum h(c)} \qE_{h(c)}([n-k-1]_q).
\end{equation}
Here one has to be a little cautious about the substitution of the
inner summation in \eqref{qZ_step} by the $q$-Ehrhart polynomial
$\qE_{h(c)}$. This is a priori allowed only if $n \geq k+1$. On the
one hand, if $2 \leq n \leq k$, then the inner summation in
\eqref{qZ_step} vanishes because it runs over an empty set. On the
other hand, by Ehrhart reciprocity, the value $\qE_{h(c)}([-d]_q)$
vanishes when $1 \leq d \leq k-1$, because in this case the $d$-th
dilate of the standard basic simplex with $k$ vertices has empty interior.

By the appropriate shift of variables relating $[n]_q$ and
$[n-k-1]_q$, one finally reaches the following definition, consistent
with all previous formulas.
\begin{definition}
  The $q$-Zeta polynomial $\qZ_{P,h}$ of a finite poset $P$ with respect to the height
  function $h$ is given by
  \begin{equation}
    \label{def_qZ}
    \qZ_{P,h}(x) = \sum_{k \geq 1}\sum_{c\in \Ch_k(P)} q^{\sum h(c)} \qE_{h(c)}\left(\frac{x-[k+1]_q}{q^{k+1}}\right).
  \end{equation}
  This polynomial is an element of the ring $\QQ(q)[x]$. Its degree is
  the maximal value of the height function $h$ on $P$. It
  belongs to the sub-ring of polynomials whose values at every
  q-integer $[n]_q$ with $n \geq 2$ is a polynomial in $q$ with positive integer
  coefficients.
\end{definition}
The first and second properties are clear from Formula \eqref{def_qZ} and the general
properties of the $q$-Ehrhart polynomials. The third property follows directly from \eqref{qZ_step}, which holds by construction, as well as \eqref{qZ_sum_over_chains}.

In other words, by \eqref{qZ_sum_over_chains}, one has the following statement:
\begin{lemma}
  The values of $\qZ_{P,h}$ at $q$-integers $[n]_q$ for $n \geq 2$ are
  $q$-analogues of the numbers of chains $e_1 \leq \dots \leq e_{n-1}$
  in $P$, where the power of $q$ is the sum of the heights of elements
  in the chain.
\end{lemma}

As expected, the polynomial $\qZ_{P,h}$ deserves the name of
$q$-analogue of the Zeta polynomial.
\begin{lemma}
  For any height function $h$, the specialisation of $\qZ_{P,h}$ at $q=1$ is the usual Zeta polynomial $Z_P$ of the poset $P$.
\end{lemma}
\begin{proof}
  This follows by comparing \eqref{classical_Z1} and \eqref{qZ_sum_over_chains}.
\end{proof}


\subsection{Examples}

Let us now give a few examples, using the height function $\rk$ coming from
the grading of the posets, unless stated otherwise.

\begin{example}
  \label{ex1}
  For the unique poset $\circ$ with one element with respect to $\rk$,
  the $q$-Zeta polynomial is $1$. More generally, for the height
  function on $\circ$ with value $H \in \NN$, the value of the $q$-Zeta
  polynomial at $[n]_q$ is $q^{(n-1)H}$ and therefore
  \begin{equation}
    \qZ_{\circ,H}(x) = \left(\frac{1+(q-1)x}{q}\right)^H.
  \end{equation}
\end{example}

\begin{example}
  \label{ex2}
  For the total order on $d \geq 2$ elements, one finds
  \begin{equation}
    \frac{\prod_{j=0}^{d-2}([j]_q+q^j x)}{[d-1]!_q}.
  \end{equation}
  This follows from the identification of the value at $[n]_q$, using the
  Formula \eqref{qZ_sum_over_chains} as a weighted sum over all chains,
  with the standard $q$-binomial coefficient counting lattice paths in a
  $(n-1)\times(d-1)$ rectangle according to the area below.
\end{example}

\begin{example}
  \label{ex3}
  For the graded poset on 5 elements with one minimum $\bott$, one
  maximum $\topp$ and 3 pairwise incomparable elements in between, one finds
  \begin{equation}
    \frac{x \cdot (\left(q + 2\right) x - 1)}{q + 1}.
  \end{equation}
  This was found using a computer. By hand, it can be computed by
  interpolation, as the degree is known to be $2$. This is made easier
  by the statements about the values at $[0]_q$ and $[1]_q$ given below in \Cref{valeur_0} and \Cref{valeur_1}.
\end{example}

\begin{example}
  \label{ex4}
  Let us also consider the poset on $\{a,b,c,d\}$ where $a$ and $b$ are
  both smaller than both $c$ and $d$. One obtains
  \begin{equation}
    \frac{2(q + 1)}{q} (x - 1),
  \end{equation}
  whose value at $[1]_q$ is $0$.
\end{example}

\begin{example}
  \label{ex5}
  For the poset on $\{a,b,c\}$ where $a$ is smaller than $b$ and $c$, one obtains
  \begin{equation}
    2 x-1,
  \end{equation}
  whereas for the dual poset one gets
  \begin{equation}
    \frac{(\left(q + 1\right) x - 1)}{q}.
  \end{equation}
\end{example}

\begin{example}
  \label{ex6}
  For the poset on $\{a,b,c,d,e\}$ where $a \leq b, a\leq c, b\leq d, c\leq e$, one obtains
  \begin{equation}
    \frac{2q x^{2} + 2 x - q - 1}{q + 1}.
  \end{equation}
\end{example}


\section{Basic properties}

\label{sect2}

Let us present in this section a few basic properties of $\qZ_{P,h}$.

\smallskip

Let $P$ and $Q$ be two posets with height functions $g$ and
$h$. Consider $P\times Q$ with the height function $g+h$ sending
$(a,b)$ to $g(a)+h(b)$. Consider also $P \sqcup Q$ with the height
function $g\sqcup h$ defined by $g$ on $P$ and $h$ on $Q$.

\begin{lemma}
  With the notations above, $\qZ_{P\times Q, g+h} = \qZ_{P,g} \qZ_{Q,h}$ and
  $\qZ_{P \sqcup Q, g \sqcup h} = \qZ_{P,h} + \qZ_{Q,h}$.
\end{lemma}
\begin{proof}
  This is most easily seen using Formula \eqref{qZ_sum_over_chains}.

  A chain $e$ in the Cartesian product $P \times Q$ is the same as a
  pair of chains $(e',e'')$ in $P$ and $Q$. The sum over $e$ of the
  heights in $P \times Q$ is that of $e'$ in $P$ plus that of $e''$ in
  $Q$.

  For the disjoint union $P\sqcup Q$, chains are either entirely in $P$ or
  entirely in $Q$ and the result follows.
\end{proof}

\medskip

For a poset $P$, let $\chi_P$ denote the Euler characteristic of the
order complex of $P$, which is the simplicial complex made of strict
chains. Recall that if $P$ has a unique minimum or maximum, the order
complex is contractible.

\begin{lemma}
  \label{valeur_1}
  The value of $\qZ_{P,h}$ at $[1]_q$ is $\chi_P$. In particular,
  if $P$ has a unique minimum or maximum, this is $1$.
\end{lemma}
\begin{proof}
  Let us first compute $\qE_{h(c)}([-k]_q)$ for $k \geq 1$ and
  $c\in \Ch_k(P)$. By Ehrhart reciprocity,
  $\qE_{h(c)}([-k]_q) = (-1)^{k-1} q^{-\sum h(c)}$ because the only
  interior point in the $k$-th dilate of the standard basic simplex is
  $(1,\ldots,1)$. Then using Formula \eqref{qZ_almost}, one finds
  \begin{equation*}
   \qZ_{P,h}([1]_q) =  \sum_{k \geq 1} (-1)^{k-1} \# \Ch_k(P).
  \end{equation*}
  This is exactly the
  expected Euler characteristic.
\end{proof}

\begin{lemma}
  \label{valeur_0}
  Assuming that $P$ has a unique minimum $\bott$, the value of
  $\qZ_{P,h}$ at $[0]_q$ is
  $q^{-h(\bott)} (1 - \chi_{P\setminus \bott})$.  In particular, if
  $P$ also has a unique maximum $\topp$ distinct from $\bott$, this is $0$.
\end{lemma}
\begin{proof}
  The special case when $P$ has only one element is clear, see \Cref{ex1}.

  Assume now that $P$ has at least $2$ elements. First one can show that
  \begin{equation*}
    \qE_{h(c)}([-k-1]_q) = (-1)^{k-1} q^{-\sum h(c)} \sum_i q^{-h(c_i)},
  \end{equation*}
  using Ehrhart reciprocity, for $k \geq 1$ and $c\in
  \Ch_k(P)$. Therefore one finds
  \begin{equation*}
    \qZ_{P,h}([0]_q) =  \sum_{k \geq 1} (-1)^{k-1} \sum_{c \in \Ch_k(P)} \sum_i q^{-h(c_i)}.
  \end{equation*}

  In this sum, the strict chain reduced to $\bott$ contributes the term $q^{-h(\bott)}$.

  On the remaining chains, removing or adding $\bott$ define bijections exchanging
  strict chains with $k+1$ elements containing $\bott$ and strict
  chains with $k$ elements not containing $\bott$. The contributions
  of such a pair of chains to the previous sum almost cancel mutually,
  except one summand $q^{-h(\bott)}$ in the inner sum, contributing
  \begin{equation*}
    \sum_{k \geq 1} (-1)^{k}  \sum_{c \in \Ch_{k}(P\setminus \bott)} q^{-h(\bott)},
  \end{equation*}
  in which one can recognize $- q^{-h(\bott)} \chi_{P\setminus \bott}$.
\end{proof}

Recall that a poset is \textbf{bounded} if it has unique minimum $\bott$ and unique maximum $\topp$.

\begin{lemma}
  \label{valeur_moins1}
  If the poset $P$ is bounded, then
  \begin{equation}
    \qZ_{P,h}([-1]_q) =  q^{-h(\bott)-h(\topp)} \mu_P(\bott,\topp),
  \end{equation}
  where $\mu_P$ is the usual Möbius function of the poset.
\end{lemma}
\begin{proof}
  The proof is very similar to that of the previous lemmas.

  The special case when $P$ has only one element is clear, see
  \Cref{ex1}. One can therefore assume that $\bott \not= \topp$.

  Using Ehrhart reciprocity to evaluate $\qE_{h(c)}$ at
  $[-k-2]_q$ for $c\in \Ch_k(P)$, one gets the formula
  \begin{equation*}
    \sum_{k \geq 1} (-1)^{k-1} \sum_{c \in \Ch_k(P)} \sum_i \left(q^{-2 h(c_i)} +\sum_{i \not = j} q^{-h(c_i)-h(c_j)} \right).
  \end{equation*}
  Then one first separates strict chains according to whether they
  start by $\bott$ or not. Contributions of pairs of chains almost
  cancel. In the remaining sum, one separates strict chains according
  to whether they end by $\topp$ or not. Once again, there are
  cancellations by pairs. There remains only
  \begin{equation*}
    q^{-h(\bott)-h(\topp)} \left(-1 + \sum_{k \geq 1} (-1)^{k-1} \# \Ch_k(P')\right),
  \end{equation*}
  where $\# \Ch_k(P')$ is the number of strict chains
  $c_1 < c_2 < \cdots < c_k$ in the poset
  $P' = P \setminus \{\bott,\topp\}$. By a classical result, the reduced
  Euler characteristic between the parentheses is the Möbius number $\mu_P(\bott,\topp)$.
\end{proof}

Let us consider a poset $P$ and its dual poset $\barP$. To any
height function $h$ on $P$ and for any integer $H$ at least equal to
the maximal value of $h$, the function $H-h$ is a height function on
$\barP$. In that case, the $q$-Zeta polynomial of $P$ w.r.t. $h$ and
that of $\barP$ w.r.t. $H-h$ are not directly related in a simple way, but
their values are related as follows.
\begin{lemma}
  \label{duality}
  One has the following relation:
  \begin{equation}
    \qZ_{\barP,H-h}([n]_q) = q^{(n-1)H} \left( \qZ_{P,h}([n]_q) \right)|_{q=1/q},  
  \end{equation}
  for all $n \in \ZZ$. Note that the replacement of $q$ by $1/q$ also
  affects the coefficients of the polynomial.
\end{lemma}
\begin{proof}
  For $n \geq 2$, this follows directly from
  \eqref{qZ_sum_over_chains}. The case $n=1$ also follows from
  \Cref{valeur_1}. It is therefore enough to check that the right hand
  side is the value of a polynomial at the $q$-integer
  $[n]_q$. Indeed, let us write
  \begin{equation*}
    \qZ_{P,h}(x) = \sum_{j=0}^{H} \kappa_j x^j,
  \end{equation*}
  as the degree of this polynomial is at most $H$. Then the right hand side in the expected relation is
  \begin{equation*}
    \sum_{j=0}^{H} \kappa_j|_{q=1/q} (q^{n-1})^{H-j} [n]_q^j,
  \end{equation*}
  which is the value at $[n]_q$ of the polynomial
  \begin{equation*}
    \sum_{j=0}^{H} \kappa_j|_{q=1/q} \left(\frac{1 + (q - 1) x}{q}\right)^{H-j} x^j.
  \end{equation*}
\end{proof}

See \Cref{ex5} for a simple case of this relationship between $q$-Zeta polynomials of dual posets.

\begin{lemma}
  Let $P$ be a poset with height function $h$. Let $\qZ_{P,h}(x)$ be the
  corresponding $q$-Zeta polynomial. Then the $q$-Zeta polynomial
  $\qZ_{P,h+1}$ for $P$ with the shifted height function $h+1$ is given by
  \begin{equation}
    \qZ_{P,h+1}(x) = \frac{1}{q} (1 + (q-1) x) \cdot \qZ_{P,h}(x).
  \end{equation}
\end{lemma}
\begin{proof}
  Using Formula \eqref{qZ_sum_over_chains}
  for the height function $h+1$ gives directly
  \begin{equation*}
    \qZ_{P,h+1}([n]_q) = q^{n-1} \qZ_{P,h}([n]_q),
  \end{equation*}
  from which the result follows.
\end{proof}

\begin{lemma}
  Let $P$ be a poset with height function $h$. Let $\qZ_{P,h}(x)$ be
  the corresponding $q$-Zeta polynomial. Let $D \geq 1$ be an
  integer. Then the $q$-Zeta polynomial $\qZ_{P,D h}$ for $P$ with the
  scaled height function $D h$ is given by
  \begin{equation}
    \qZ_{P,D h}(x) = \qZ_{P,h}\bigm|_{q=q^D}\left(\frac{(1+(q-1)x)^D-1}{q^D-1}\right).
  \end{equation}
\end{lemma}
\begin{proof}
  From \eqref{qZ_sum_over_chains}, it follows that
  $\qZ_{P,D h}([n]_q) = \qZ_{P,h}|_{q=q^D}([n]_{q^D})$ for $n \geq 2$. Then
  expressing the argument $[n]_{q^D}$ using $[n]_q$ gives the formula.
\end{proof}

This property can be seen in \Cref{ex1}.

\smallskip


Recall that the coefficients of the flag $f$-vector of a graded poset
$P$ are the numbers of strict chains $c$ with a fixed sequence of
ranks $\rk(c)$. For more on this, see \cite{billera_flag} and \cite[\S 3.12]{stanleyEC1}.

\begin{lemma}
  Let $P$ be a graded poset and consider the height function $\rk$.
  The polynomial $\qZ_{P,\rk}$ is entirely determined by the flag $f$-vector of $P$.
\end{lemma}
\begin{proof}
  It follows from Formula
  \eqref{qZ_almost} that $\qZ_{P,\rk}$ is a linear combination of
  $q$-Ehrhart polynomials $\qE_{a}$ whose coefficients are exactly
  coefficients of the flag $f$-vector.
\end{proof}

Note: one can wonder about the converse implication. Most probably,
the flag $f$-vector should be a finer invariant than the $q$-Zeta
polynomial.

\section{Relation with the $q$-order polynomial}

\label{sect3}

The order polytope $Q_P$ of a poset $P$ \cite{stanley_two_polytopes} is the lattice polytope in
$\NN^{P}$ defined by inequalities $0 \leq z_p \leq 1$ for all
$p \in P$ and $z_p \leq z_q$ if $p \leq q$ in $P$. The application
$(z_p)_p \mapsto (1-z_p)_p$ is a bijection between $Q_P$ and
$Q_{\barP}$ for the dual poset $\barP$.

\smallskip

Let us recall the $q$-order polynomial $\qL_P$ of a poset $P$ as introduced in \cite{q-ehrhart}. This
is the $q$-Ehrhart polynomial of the order polytope $Q_P$, with
respect to the linear form $(1,1,\ldots,1)$, \textit{i.e.} the sum of coordinates. The value of $\qL_P$ at
$[n]_q$ for $n \geq 0$ is therefore
\begin{equation}
  \qL_{P}([n]_q) = \sum_{z \in n Q_P} q^{\sum_p z_p}.
\end{equation}
The degree of $\qL_P$ is $\# P$.

Let $P$ be a finite poset. Let $J(P)$ be the distributive lattice of
lower ideals in $P$ under the order of containment. This lattice is
graded by the cardinality of the lower ideal. Fix $n \geq 2$, and
consider a chain $e_1\leq \ldots\leq e_{n-1}$ of elements of
$J(P)$. This is an increasing chain of lower ideals of $P$. As such,
it is characterised by the following data : for each element $p$ of
$P$, let $z_p$ be the smallest integer between $0$ and $n-2$ such that
$p \in e_{z_p+1}$ if it exists and $n-1$ otherwise. Then these vectors
$(z_p)_{p \in P}$ satisfy that $z_p \leq z_q$ if $p \leq q$ in
$P$. This map defines a bijection between chains
$e_1\leq \ldots\leq e_{n-1}$ in $J(P)$ and lattice points in the
$(n-1)$-dilate of the order polytope $Q_P$.

\begin{proposition}
  \label{Z_of_J}
  For every poset $P$, one has $\qZ_{J(P),\rk}(1 + q x) =  \qL_{\barP}(x)$.
\end{proposition}
\begin{proof}
  Let us compare the values at $x = [n-1]_q$ for $n \geq 2$. Let us compute
  the monomial in $q$ attached in
  $\qZ_{J(P),\rk}(1 + q x) = \qZ_{J(P),\rk}([n]_q)$ to one chain
  $$e_1\leq \ldots\leq e_{n-1}$$ in terms of the corresponding lattice
  point $(z_p)_{p\in P}$ in $(n-1)Q_{P}$. For the chain, the monomial 
  is $q$ to the power $\sum_i \# e_i = \sum_i \sum_{p \in e_i} 1$. By the
  bijection between chains and lattice points, the exponent of $q$
  becomes $\sum_{p \in P} (n - 1 - z_p)$ where $z \in (n-1)
  Q_P$. Using the bijection $(z_p)_p \mapsto (1-z_p)_p$ between $Q_P$
  and $Q_{\barP}$, the exponent of $q$ becomes
  $\sum_{p \in \barP} z_p$ where $z \in (n-1) Q_{\barP}$. Summing
  these monomials over $z$ gives exactly the value at $[n-1]_q$ of the
  $q$-order polynomial of $\barP$. 
\end{proof}

For example, let $P$ be the poset with three elements $a$, $b$ and $c$, such that $a$
is less than $b$ and $c$. Then the $q$-Ehrhart polynomial of $\barP$ is
\begin{equation*}
  \frac{1}{[2]_q [3]_q} \cdot (q x + 1) \cdot (q^{2} x + q + 1) \cdot (\left(q^{3} + q^{2}\right) x + q^{2} + q + 1),
\end{equation*}
and the $q$-Zeta polynomial of $J(P)$ is
\begin{equation*}
   \frac{1}{[2]_q [3]_q} \cdot x \cdot (q x + 1) \cdot (\left(q^{2} + q\right) x + 1).
\end{equation*}

\section{Ehrhart series and volumes}

\label{sect4}

\subsection{$P$-partitions and $q$-Ehrhart series for posets}

\label{section-p-partition}

Let us first recall the following classical setting, part of the more
general theory of $P$-partitions, due to Stanley \cite[\S
8]{stanley_ordered}.

The $q$-Ehrhart series of any poset $P$ is the formal power series defined by
\begin{equation}
  \label{q-ehrhart-series}
  \qserieL_P = \sum_{n \geq 0} \qL_P([n]_q) t^n,
\end{equation}
and can be expressed as a rational fraction
\begin{equation}
  \qserieL_P = \frac{\HH_P(q, t)}{(1- t)(1-q t)\dots(1-q^{\#P} t)},
\end{equation}
where $\HH_P$ is a polynomial in $q$ and $t$ with non-negative integer
coefficients. This polynomial has a known combinatorial
interpretation, using descents for $t$ and major indices for $q$, as a
sum over all linear extensions of the poset $P$.

Let us also introduce the $q$-volume of a poset, as defined in \cite[\S 4.2]{q-ehrhart}.
\begin{definition}
  \label{def-q-vol}
  The $q$-volume of a poset $P$ is the leading coefficient of the
  $q$-Ehrhart polynomial $\qL_P$ times the $q$-factorial $[\# P]!_q$.
\end{definition}

Proposition 4.9 in \cite{q-ehrhart} gives the following relationship between the polynomial $\HH_{\barP}$ and the $q$-volume of $P$.

\begin{proposition}
  The $q$-volume of $P$ is equal to $q^{\binom{\# P + 1}{2}}$ times $\HH_{\barP}(1/q, 1)$.
\end{proposition}

\subsection{$q$-Ehrhart series for $q$-Zeta polynomials}

Inspired by \Cref{Z_of_J} which identifies the $q$-Zeta polynomials of
distributive lattices with $q$-order polynomials of their posets of
join-irreducibles, it is natural to extend the constructions of the
previous paragraph to the general case, for the $q$-Zeta polynomials
of arbitrary posets.

Let $(P,h)$ be a poset endowed with a height function $h$. In this
section, $H$ will denote the maximal value of $h$.

Let us consider the generating series of values of $\qZ_{P,h}$:
\begin{equation}
  \label{defi:serie_Z}
  \qserieZ_{P,h} = \sum_{n \geq 0} \qZ_{P,h}([n+1]_q) t^n,
\end{equation}
similar to the $q$-Ehrhart series \eqref{q-ehrhart-series}.

\begin{proposition}
  \label{defi:HH}
  This series can be expressed as a rational fraction:
  \begin{equation}
    \qserieZ_{P,h} = \frac{\HH_{P,h}(q,t)}{(1- t)(1-q t)\dots(1-q^{H} t)},
  \end{equation}
  where $\HH_P$ is a polynomial in $q, q^{-1}$ and $t$ with integer
  coefficients. The degree of $\HH_{P,h}$ with respect to $t$ is at most $H$.
\end{proposition}
\begin{proof}
  Let us write $\qZ$ for $\qZ_{P,h}$.
  Consider the polynomial $\qZ(1 + qx)$. By \Cref{qZ_sum_over_chains}
  and \Cref{valeur_1}, its value at every $[n]_q$ for $n \geq 0$ is a
  polynomial in $q$ with integer coefficients. By \Cref{basis_Aq}, it
  can therefore be expressed as a sum
  \begin{equation}
    \qZ(1 + qx) = \sum_{0 \leq j \leq H} c_j B_j(x),
  \end{equation}
  where $c_j \in \ZZ[q, q^{-1}]$ and $B_j$ are polynomials defined in Appendix \ref{app:q-ring} by \eqref{defi:B}. 
  Then one concludes using the generating series from \Cref{serie_B}.
\end{proof}

By analogy with \Cref{def-q-vol}, let us introduce the volume associated with the $q$-Zeta polynomial.
\begin{definition}
  \label{def-q-z-vol}
  The $q$-Zeta volume of a poset $P$ with respect to the height function $h$ is
  the leading coefficient of the $q$-Zeta polynomial $\qZ_{P,h}$ times
  the $q$-factorial $[H]!_q$.
\end{definition}

\medskip

In this context, one has the following general property, for arbitrary
posets. Let us choose $H-h$ as the height function on the dual poset
$\barP$.
\begin{proposition}
  \label{volume_and_HH}
  The $q$-Zeta volume of $P$ w.r.t. $h$ is $q^{\binom{H}{2}}$ times $\HH_{\barP, H-h}(1/q,1)$.
\end{proposition}
\begin{proof}
  The proof is similar to the proof of \cite[Prop. 4.9]{q-ehrhart}. Let us write
  \begin{equation}
    \label{expression_series}
    \qserieZ_{P,h} = \frac{\sum_{j=0}^{H} h_j t^j}{\prod_{\ell=0}^{H} 1-q^\ell t},
  \end{equation}
  for some coefficients $h_j$ in $\ZZ[q,q^{-1}]$.

  By \Cref{duality} translated into generating series using \eqref{defi:serie_Z}, one has the equality
  \begin{equation*}
    \HH_{\barP, H-h}(q, t) = \HH_{P,h}(1/q,q^{H}t).
  \end{equation*}
  After replacing $q$ by $1/q$ in the evaluation at $t=1$, one gets
  \begin{equation*}
    \HH_{\barP, H-h}(1/q, 1) = \sum_{j=0}^{H} h_j q^{-j H}.
  \end{equation*}

  On the other hand, one can compute directly the coefficient
  of $x^H$ in $\qZ_{P,h}(1+qx)$ from \eqref{expression_series}. For the term
  $j=0$, one can use \Cref{serie_B} and \eqref{defi:B} that gives the
  corresponding coefficient of $x^H$ explicitly as
  \begin{equation*}
    \frac{1}{[H]!_q}\prod_{j=1}^{H} q^j =  \frac{q^{\binom{H+1}{2}}}{[H]!_q}.
  \end{equation*}

  For the terms of index
  $j \geq 1$, one notes that the product by $t$ in the series amounts
  to replace $x$ by $(x-1)/q$ in the polynomial coefficient, which
  multiplies the leading coefficient of $x^H$ by $q^{-H}$.  The total
  leading coefficient of $\qZ_{P,h}(1+x)$ is therefore given by
  \begin{equation*}
    \left( \sum_{j=0}^{H} h_j q^{-j H} \right) \frac{q^{\binom{H+1}{2}}}{[H]!_q}.
  \end{equation*}
  For the leading coefficient of $\qZ_{P,h}(x)$, one has moreover to
  divide by $q^H$.
  
  One concludes by comparing the obtained expressions.
\end{proof}

For example, for the self-dual poset of \Cref{ex3} with $H=2$, the series is
\begin{equation*}
  \qserieZ_{P,\rk} = \frac{1+2 q t}{(1-t)(1-qt)(1-q^2 t)}
\end{equation*}
and the $q$-volume is $q+2$.

\section{Positivity properties}

\label{sect5}

By \Cref{Z_of_J} and the discussion in \S \ref{section-p-partition}, when $P$ is a
distributive lattice, then $\HH_{P,\rk}$ is a polynomial in $q$ and
$t$ with non-negative integer coefficients.

This positivity property of $\HH_{P,\rk}$ is not true for all posets. Small
counterexamples are the posets of \Cref{ex5} and \Cref{ex6}.


In the case of \Cref{ex6} with $H=2$, one gets the numerator
\begin{equation*}
  -q^3 t^{2} + \left(2 + q + q^{2}\right) t - 1.
\end{equation*}


In this section, we will give a sufficient criterion for positivity of
$\HH_{P,\rk}$, in terms of the existence of an $R$-labelling.

\subsection{$R$-labellings}

\label{r-label}

Let $L$ be a set of labels, endowed with an arbitrary relation denoted by
$\rel$. For $p < q$ in $P$, a maximal chain
$p = e_0 < e_1 < \cdots < e_{n+1} = q$ is \textit{increasing}\footnote{This terminology comes from the case where $\rel$ is a partial order relation, which is not assumed here. One may say concatenable of friendly for a better intuition.} if
$\lambda(e_{i-1},e_{i}) \rel \lambda(e_{i},e_{i+1})$ for all $1 \leq i \leq n$.

\smallskip

An $R$-labelling of a poset $P$ by $(L, \rel)$ is a function
$\lambda$ from the set of edges of the Hasse diagram of $P$ to $L$
such that:
\begin{itemize}
\item for every pair of comparable elements $p \leq q$ in $P$, there is exactly one increasing maximal chain from $p$ to $q$.
\end{itemize}

\medskip

For example, the weak order on the symmetric group $S_3$ has no $R$-labelling, as both maximal chains from the minimum to the maximum are necessarily increasing.

\begin{remark}
  Every $EL$-labelling, as defined in
  \cite{bjorner-shellable}, is also an $R$-labelling, as it satisfies a
  stronger condition. Supersolvable lattices and upper-semimodular lattices
  always have an $R$-labelling. This is proved in \cite[Examples 3.13.4
  and 3.13.5]{stanleyEC1}.
\end{remark}


\subsection{Positivity criterion}

\begin{proposition}
  \label{HH_is_positive}
  If $P$ is a bounded and graded poset that admits an $R$-labelling,
  then $\HH_{P,\rk}$ is a polynomial with positive integer coefficients.
\end{proposition}
\begin{proof}
  The proofs follows from \Cref{thm:if_rlabel_positive} that gives
  positivity of the flag $h$-vector under the given hypothesis and the
  fact that the coefficient of $\HH_{P,\rk}$ are non-negative linear
  combinations of the flag $h$-vector elements, as proved in
  \Cref{thm:H_and_beta}.
\end{proof}

\begin{lemma}
  If $P$ is a bounded and graded poset that admits an $R$-labelling,
  then the $q$-Zeta volume of $\qZ_{P,\rk}$ is a non-negative $q$-analogue of the
  number of maximal chains in $P$.
\end{lemma}
\begin{proof}
  This follows from \Cref{HH_is_positive} and \Cref{volume_and_HH}. The
  relationship between the $q$-Zeta volume and $\HH_{P,\rk}$ involves the
  dual poset $P$, but the set of maximal chains is preserved by
  duality.
\end{proof}

\medskip

\Cref{HH_is_positive} above applies to several classical posets
attached to finite Coxeter groups. The first ones are non-crossing
partitions lattices issued from the absolute order \cite{brady-watt,
  brady-watt-Kpi1,bessis,mccammond} and shard-intersection orders
\cite{reading-fpsac,reading}. These two families are known to be
$EL$-shellable by results of \cite{athaBW,petersen}.

The intersection lattice of an essential central hyperplane
arrangement is a geometric lattice and is therefore EL-shellable
\cite[Th. 3.1]{bjorner-shellable}. This applies in particular to posets of
generalized set partitions defined as intersection lattices of
reflection hyperplane arrangements.



Hence, all these posets have an $R$-labelling and non-negative $\HH_{P,\rk}$.

\smallskip

There is a less well-known family of posets attached to finite Coxeter
groups, the parabolic-support posets \cite{BHZ, posets_BHZ}. Every
interval in this family is shellable and upper-semimodular, hence has
an $R$-labelling \cite{BHZ}. But these posets are not bounded, hence
the criterion does not apply to the full posets. Positivity of
$\HH_{P,\rk}$ seems nevertheless to hold for the parabolic-support
posets. This remains to be proved and explained.

\smallskip

Another interesting example is given by the root posets of type $B$,
which are not bounded posets either. One can endow these posets with
an $R$-labelling according to the direction (North-East or North-West)
of the cover relations. It seems that the coefficients of powers $t^k$
in $\HH_{P,\rk}$ are non-negative $q$-analogues of $\binom{n}{k}^2$,
the Narayana numbers of type $B$.

\medskip

One can wonder what would be a necessary and sufficient condition on
posets for positivity of the coefficients of $\HH_{P,\rk}$.






\section{Value at $q=0$ and characteristic polynomial}

\label{sect6}

Assume in this section that $P$ is bounded and graded with a unique
minimal element $\bott$ and a unique maximal element $\topp$. Let $H$
be the maximal value of the rank function $\rk$ on $P$.

The characteristic polynomial of $P$ is defined as
\begin{equation}
  X_P(y) = \sum_{p \in P} \mu(\bott, p)\, y^{H-\rk(p)}.
\end{equation}

We will use the notions of flag $f$-vector and flag $h$-vector
recalled in \Cref{app:flags} and the notations introduced there.

The characteristic polynomial can be expressed using the flag
$h$-vector as follows.
\begin{lemma}
  \label{char_and_beta}
  For every graded and bounded poset $P$,
  \begin{equation}
    y^{H} X_P(1/y) = 1 + \sum_{j=1}^{H} \left( \beta_P([1,\ldots,j-1]) + \beta_P([1,\ldots,j]) \right) (-y)^j.
  \end{equation}
\end{lemma}
\begin{proof}
  The polynomial on the left is
  \begin{equation*}
    1 + \sum_{j=1}^{H} \left(\mathop{\sum_{p \in P}}_{rk(p)=j} \mu(\bott, p)\right) y^j.
  \end{equation*}
  Then, for $j \geq 1$, the coefficient of $y^j$ is the difference
  \begin{equation*}
    \mathop{\sum_{p \in P}}_{rk(p)\leq j}  \mu(\bott, p) - \mathop{\sum_{p \in P}}_{rk(p)\leq j-1} \mu(\bott, p),
  \end{equation*}
  in which each sum is minus the Möbius number $\mu_S$ of a
  rank-selected sub-poset, for the sets $S=[1,\ldots,j]$ and
  $S=[1,\ldots,j-1]$. The last step is then to use the relationship
  $\mu_S = (-1)^{\# S - 1} \beta_P(S)$ (\cite[3.12]{stanleyEC1}).
\end{proof}


\begin{theorem}
  The $q$-Zeta polynomial $\qZ_{P,\rk}$ has no pole at $q=0$ and its value at $q=0$ is related to the characteristic polynomial by
  \begin{equation}
    \qZ_{P,\rk}|_{q=0}(1-y) = y^H X_P(1/y).
  \end{equation}
\end{theorem}

\begin{remark}
  It can well happen, when the poset $P$ has no unique minimum, that
  $\qZ_{P,h}$ has poles at $q=0$, for instance for the poset of
  \Cref{ex4} and the dual poset in \Cref{ex5}.
\end{remark}

\begin{proof}
  Let us denote $\qZ_{P,\rk}$ by $\qZ$.
  Let us start with \Cref{thm:H_and_beta}. One deduces that
  \begin{equation*}
    t \HH_{P,\rk}(q,t) = \sum_{S \subseteq \{1,\ldots,H-1\}} \beta_P(S) t^{\#S+1} q^{\Sigma S}.
  \end{equation*}
  As $P$ is bounded, $\qZ(0) = 0$ by \Cref{valeur_0}. Therefore the
  previous expression is the numerator of the generating series of
  values of $\qZ$ for $n \geq 0$. By the correspondence stated
  in appendix \ref{app:q-ring}, this means that
  \begin{equation*}
    \qZ_{P,\rk}(x) = \sum_{S \subseteq \{1,\ldots,H-1\}} \beta_P(S) q^{\Sigma S} \qbinom{H-\#S - 1 ; x}{H}.
  \end{equation*}
  One can then use \Cref{limit_q0} to let $q$ be $0$ and obtain
  \begin{equation*}
    \qZ|_{q=0}(x) = \sum_{j=0}^{H-1} \beta_P([1,\ldots,j]) x(x-1)^j.
  \end{equation*}
  Indeed, all sets $S$ that are not formed of consecutive integers
  starting at $1$ appear with a strictly positive power of $q$, hence
  vanish when $q=0$. One deduces
  \begin{equation*}
    \qZ|_{q=0}(1-y) = \sum_{j=0}^{H-1} \beta_P([1,\ldots,j]) (1-y)(-y)^j.
  \end{equation*}
  One can then conclude by an easy comparison with \Cref{char_and_beta}.
\end{proof}












\section{Reciprocity for Eulerian posets}

\label{sect7}

Recall from \cite[\S 3.14]{stanleyEC1} that a finite graded poset $P$
with $\bott$ and $\topp$ is Eulerian if its Möbius function $\mu$ is
given for all $x \leq y$ in $P$ by $\mu(x,y) = (-1)^{\rk(x,y)}$ where
$\rk$ is the rank function of $P$.

The following proposition is a $q$-analogue of a classical reciprocity
property for Zeta polynomials of Eulerian posets, see \cite[\S
2]{stanley_eulerian} and \cite[Prop. 3.14.1]{stanleyEC1}.

We will use in the proof the notations of \Cref{app:q-incidence} and
work in $\Inc_q(P,\rk)$.

\begin{proposition}
  Let $P$ be an Eulerian poset endowed with its rank function $\rk$
  and let $n$ be the rank of $\topp$. Then
  \begin{equation*}
    \qZ(x) = (-q)^n \qZ|_{q=1/q} (-qx).
  \end{equation*}
\end{proposition}
\begin{proof}
  It is enough to prove the equality when evaluated at any negative
  $q$-integer $[-m]_q$ for $m \geq 2$. This becomes
  \begin{equation}
    \label{step_eulerian}
    \qZ([-m]_q) = (-q)^n \left( \qZ([m]_q)  \right)|_{q=1/q}.
  \end{equation}
  Recall from \Cref{app:q-incidence} the standard Zeta matrix $Z$ of the poset
  and the diagonal matrix $D_{\rk}$ whose coefficients are $q^{\rk(e)}$
  for $e \in P$. Let $M$ be the standard Möbius matrix defined by $M = Z^{-1}$.

  By the description given in \Cref{app:q-incidence}, the left hand
  side of \eqref{step_eulerian} is the corner coefficient in the
  inverse in $\Inc_q(P,\rk)$ of the matrix
  \begin{equation*}
    Z (D_{\rk} Z)^{m-1},
  \end{equation*}
  which is
  \begin{equation*}
    D_{\rk}^{-1} (M D_{\rk}^{-1})^m.
  \end{equation*}
  The corner coefficient is therefore the sum over all paths
  $\bott = e_0 \leq e_1 \leq e_2 \leq \cdots \leq e_{m} = \topp$ of 
  \begin{equation*}
    q^{0} \mu(\bott,e_1) q^{-\rk(e_1)} \mu(e_1,e_2) q^{-\rk(e_2)} \cdots q^{-\rk(e_{m-1})} \mu(e_{m-1}, \topp) q^{-n}.
  \end{equation*}
  Using that the poset is Eulerian and cancelling signs by pairs
  except the rightmost one, this is the same as $(-q)^{-n}$ times the
  sum over the same set of paths of
  \begin{equation*}
     q^{-\rk(e_1)} q^{-\rk(e_2)} \cdots q^{-\rk(e_{m-1})}.
  \end{equation*}
  This last sum is exactly the value $\qZ([m]_q$ in which $q$ was
  replaced by $1/q$.
\end{proof}

This proposition easily translates into properties of every individual
coefficient. For example, for the the face lattice of the icosahedron, where $n=4$,
the $q$-Zeta polynomial is the quotient of
\begin{equation*}
  (q^4 + 17 q^3 - 6 q^2 + 17 q + 1) x^4 - 16 ( q^3 - 2 q^2 + 2 q - 1) x^3 - 24 ( q^2 - q + 1) x^2 - 8 (q - 1) x
\end{equation*}
by $(q^2 + 1)   (q^2 + q + 1)$.

Similarly, for the associahedron of dimension $3$ with $14$ vertices, the $q$-Zeta polynomial is the quotient of
\begin{equation*}
  (q^4 + 6 q^3 + 7 q^2 + 6 q + 1) x^4 - 5 (q^3 + q^2 - q - 1) x^3 - 15 q x^2 + 5 (q - 1) x
\end{equation*}
by $(q^2 + 1)   (q^2 + q + 1)$.

\appendix

\section{Flag $f$-vectors and $h$-vectors}

\label{app:flags}

Let us recall the standard definitions of flag $f$-vectors and flag
$h$-vectors and then state theorems about their relationship with
maximal chains. A standard reference on the subject is \cite[\S
3.12]{stanleyEC1}.

\smallskip

Let $P$ be a graded and bounded poset with unique minimum $\bott$ and
unique maximum $\topp$. Let $\rk$ be the rank function on $P$, with
minimal value $0$ on $\bott$ and maximal value $H$ on $\topp$.

For a subset $S \subseteq \{1, 2, \dots, H - 1\}$ of cardinality $k$, let
$\alpha_P(S)$ be the number of chains
$\bott < p_1 < p_2 < \cdots < p_k < \topp$ in $P$ such that
$S = \{\rk(p_1),\ldots,\rk(p_k)\}$.

The \textbf{flag $f$-vector} of $P$ is the collection of numbers $\alpha_P(S)$,
indexed by subsets $S$ of $\{1, 2, \dots, H - 1\}$.

The \textbf{flag $h$-vector} of $P$ is the collection of numbers $\beta_P(S)$, also indexed by subsets of $\{1, 2, \dots, H - 1\}$, defined by Möbius inversion as follows:
\begin{equation}
  \beta_P(S) = \sum_{T \subseteq S} (-1)^{\# S- \#T} \alpha_P(T)
\quad\text{and conversely}\quad
  \alpha_P(S) = \sum_{T \subseteq S} \beta_P(T).
\end{equation}

Note that the numbers $\beta_P(S)$ are not obviously non-negative, as
their definition by Möbius inversion involves signs.

\smallskip

Assume now that $P$ admits an $R$-labelling $\lambda$ for the relation $\rel$, as defined in \S \ref{r-label}. To each maximal chain
$M : \bott = p_0 < p_1 < p_2 < \cdots < p_{H} = \topp$ in $P$, one can
associate its descent set\footnote{Again, the terminology comes from the case where $\rel$ is a partial order.}
\begin{equation}
  D(M) = \{ i \in \{1, \ldots, H-1\} \mid \text{ not }\lambda(p_{i-1}, p_i) \rel \lambda(p_1,p_{i+1}) \}.
\end{equation}

\begin{theorem}[Björner and Stanley]
  \label{thm:if_rlabel_positive}
  Let $P$ be a bounded and graded poset with an $R$-labelling. Let
  $S \subseteq \{1, 2, \ldots, H - 1\}$ be any subset the set of
  heights. The number $\alpha_P(S)$ counts maximal chains in $P$ with
  descent set contained in $S$. The number $\beta_P(S)$ counts
  maximal chains in $P$ with descent set $S$ and is therefore
  non-negative.
\end{theorem}

For the proof, see \cite[\S 3.13]{stanleyEC1} or
\cite[Th. 2.7]{bjorner-shellable}. One can check that the proof of
\cite[Th. 3.13.2]{stanleyEC1} works \textit{verbatim} without the
hypothesis that $\rel$ is a partial order relation.

Some interesting information about this statement can be found in
\cite[\S 8.1]{gessel-P}.

\medskip

There is a simple relationship between the flag $h$-vector and the
numerator of the $q$-Ehrhart series as defined in \ref{defi:HH}.

\begin{theorem}
  \label{thm:H_and_beta}
  The polynomial $\HH_{P,\rk}$ is the sum
  \begin{equation}
    \sum_{S \subseteq \{1,\ldots, H-1\}} \beta_P(S) t^{\# S} q^{\sum(S)}
  \end{equation}
  and has therefore non-negative coefficients.
\end{theorem}
\begin{proof}
  Introduce formal variables $u_0,u_1,\ldots,u_{H}$. For a chain $K$
  in $P \setminus \{\bott , \topp \}$, let $u_K$ be the product of
  $u_{\rk(p)}$ over elements $p$ in the chain. For a subset
  $S \subseteq \{1, 2, \dots, H - 1\}$, let $u_S$ be the product of $u_i$
  over elements of $S$.

  According to \cite[Ex. 3.67(b)]{stanleyEC1}, one has the formula
  \begin{equation*}
    \sum_{K} u_K = \frac{\sum_{S} \beta_P(S) u_S}{\prod_{\ell=1}^{H-1} 1-u_{\ell}},
  \end{equation*}
  where the sum over $K$ is running over all chains in $P \setminus \{\bott, \topp\}$.

  By extending the sum over $K$ to all chains in $P$, one gets
  \begin{equation*}
    \sum_{K} u_K = \frac{\sum_{S} \beta_P(S) u_S}{\prod_{\ell=0}^{H} 1-u_{\ell}},
  \end{equation*}
  with just two additional factors in the denominator accounting for the multiplicities of $\bott$ and $\topp$ in the chain.

  By specializing every $u_{\ell}$ to $q^{\ell} t$, one gets the equality
  \begin{equation*}
    \sum_{n \geq 0} \qZ_{P,\rk}([n+1]_q) t^{n} = \frac{\sum_{S} \beta_P(S) t^{\# S}q^{\sum(S)}}{\prod_{\ell=0}^{H} 1-q^\ell t}.
  \end{equation*}
  Comparing with the definition of $\HH_{P,\rk}$ as a numerator in \Cref{defi:HH},
  one obtains the expected formula.
\end{proof}


\section{Rings of polynomials with integer Laurent values}

\label{app:q-ring}

Let $A_q$ be the sub-ring of $\QQ(q)[x]$ made of polynomials $P$ such
that $P([n])_q \in \ZZ[q,q^{-1}]$ for all
$n \geq 0$. This is an analogue of the ring of integer-valued
polynomials.

For $k \geq 0$, let $B_k$ be the polynomial defined by
\begin{equation}
  \label{defi:B}
  B_k(x) = \frac{\prod_{j=1}^{k} [j]_q+q^j x}{[k]!_q}.
\end{equation}
Because the values $B_k([n]_q)$ for $n \geq 0$ are standard $q$-binomial
coefficients, every $B_k$ belongs to $A_q$. We will use the convention that $B_{-1}=0$.

Let $\Delta_q$ be the operator acting on polynomials in $x$ defined by
\begin{equation}
  \label{def_delta}
  \Delta_q(P)(x) = \frac{P(x)-P((x-1)/q)}{1+(q-1)x}.
\end{equation}
The right hand side is well defined as a polynomial because the
numerator has a root at $x=1/(1-q)$. Assuming moreover that $P \in A_q$, this
formula implies that the values $\Delta_q(P)([n]_q)$ for $n \geq 1$ are Laurent
polynomials in $q$ with integer coefficients. 

\begin{lemma}
  The operator $\Delta_q$ maps $B_{k}$ to $B_{k-1}$ for all $k \geq 0$.
\end{lemma}
\begin{proof}
  It is enough to prove that it holds for the value at every $q$-integer $[n]_q$ for $n \geq 1$. This in turn follows from the classical formula
  \begin{equation*}
    \left(\qbinom{n+k}{k} - \qbinom{n-1+k}{k}\right)/q^n = \qbinom{n+k-1}{k-1}
  \end{equation*}
  for $q$-binomial coefficients.
\end{proof}

\begin{proposition}
  \label{basis_Aq}
  The polynomials $B_k$ form a basis over $\ZZ[q, q^{-1}]$ of the ring $A_q$.
\end{proposition}
\begin{proof}
  These polynomials are linearly independent, as $B_k$ has degree
  $k$. It remains to prove that they span $A_q$ over $\ZZ[q, q^{-1}]$.
  The proof is by induction on the degree. This is clear in degree $0$
  as $B_0 = 1$.

  So let $P$ be a polynomial of degree $d > 0$ in $A_q$.  Then
  $Q = \Delta_q(P)$ is a polynomial with values in $\ZZ[q, q^{-1}]$
  for $n \geq 1$.  Moreover, if the leading coefficient of $P$ is
  $c x^d q^{\binom{d+1}{2}}/ [d]!_q$, then the leading
  coefficient of $Q$ is $c x^{d-1} q^{\binom{d}{2}}/
  [d-1]!_q$. Iterating $d$ times the operator $\Delta_q$, one
  reaches a constant polynomial, with values in $\ZZ[q, q^{-1}]$ for
  $n \geq d$ and with leading term $c$. Therefore $c \in \ZZ[q, q^{-1}]$ and the polynomial
  $P - c B_d$ belongs to $A_q$, with degree at most $d-1$. By
  induction, this difference is a linear combination over
  $\ZZ[q, q^{-1}]$ of $B_{k}$ for $0 \leq k \leq d-1$. Hence $P$ is a
  linear combination over $\ZZ[q, q^{-1}]$ of $B_{k}$ for
  $0 \leq k \leq d$.
\end{proof}

\begin{corollary}
  Let $P$ be an element of $A_q$. Then $P([n]_q) \in\ZZ[q, q^{-1}] $ for all $n \in \ZZ$. The operator $\Delta_q$ acts on $A_q$.
\end{corollary}
\begin{proof}
  The first statement holds because this is true for all basis
  elements $B_k$ by an easy computation. The second one follows from
  the first and \eqref{def_delta}.
\end{proof}

\begin{lemma}
  \label{unshift}
  For $k \geq 0$, there holds
  $B_k((x-1)/q) = q^{-k} \left(B_k(x) - B_{k-1}(x) \right)$.
\end{lemma}
\begin{proof}
  The special case $k=0$ is immediate. It is enough to prove this
  equality when evaluated at $[n]_q$ for $n \geq 1$. This
  reduces to a standard property of $q$-binomial coefficients.
\end{proof}

The following statement can be found for instance in \cite[Lem.
4.7]{q-ehrhart}.
\begin{lemma}
  \label{serie_B}
  For $k \geq 0$,
  \begin{equation}
    \sum_{n \geq 0} B_k([n]_q) t^n = \frac{1}{\prod_{\ell=0}^{k}(1-q^\ell t)}.
  \end{equation}
  
\end{lemma}

The action of $\Delta_q$ can be translated into an action on sequences of Laurent polynomials as follows.

Abusing notation, let $\Delta_q$ be the linear operator acting on
sequences $(a_n)_{n \in \ZZ}$ of Laurent polynomials by the formula
\begin{equation}
  \label{delta_on_seq}
  (\Delta_q(a))_n = (a_n - a_{n-1}) / q^n.
\end{equation}
This is the same as \eqref{def_delta} when $a_n = P([n]_q)$.

\begin{lemma}
  \label{annihile}
  A sequence of Laurent polynomials $(a_n)_{n \geq 0}$ is the sequence
  of values of an element of $A_q$ at q-integers $[n]_q$ with
  $n \geq 0$ if and only if it is annihilated by some power of
  $\Delta_q$.
\end{lemma}
\begin{proof}
  Assume first that $a_n= P([n]_q)$ for some element $P$ of $A_q$ and
  for all $n \geq 0$. Then because the operator $\Delta_q$ acts by
  sending the basis element $B_{k}$ to $B_{k-1}$, iterating $\Delta_q$
  sufficiently gives the constant sequence $0$.

  Conversely, one proceeds by induction on $d$ where the sequence $a$
  is annihilated by the power $d+1$ of $\Delta_q$. If the sequence $a$
  is identically zero, the statement is obvious.

  Otherwise, the sequence $\Delta_q^{d}(a)$ is a constant Laurent
  polynomial $c$. The sequence $a'_n = a_n - c B_d([n]_q)$ is then
  annihilated by $\Delta_q^{d}$, hence by induction $a'(n)= P([n]_q)$
  for some element $P$ of $A_q$. The statement follows.
\end{proof}

For integers $a$ and $b$, let us introduce the following polynomials
\begin{equation}
  \label{def:qbi}
  \qbinom{a ; x}{b} = \frac{\prod_{j=a-b+1}^{a} [j]_q + q^j x} {[b]_q!}.
\end{equation}
Their evaluation at $x=[n]_q$ is the standard $q$-binomial coefficient
$\qbinom{a+n}{b}$. Note that $B_k(x) = \qbinom{k;x}{k}$ for $k \geq 0$.

These polynomials are useful in the following correspondence.

\begin{proposition}
  Consider the generating series
  \begin{equation}
    \sum_{n \geq 0} Z_n t^n = \frac{\sum_{k=0}^{d} h_{k} t^{k} }{\prod_{\ell=0}^{d}(1-q^{\ell} t)},
  \end{equation}
  where $h_k$ are arbitrary coefficients. Then the $Z_n$ are the
  values at $[n]_q$ of the polynomial
  \begin{equation}
    \sum_{k=0}^{d} h_k \qbinom{d-k ; x}{d}.
  \end{equation}
\end{proposition}
\begin{proof}
  Starting from \Cref{serie_B} corresponding to $k=0$, one can obtain
  by induction on $k$ the equality
  \begin{equation*}
    \sum_{n \geq 0} \qbinom{d-k+n}{d} t^n = \frac{ t^{k} }{\prod_{\ell=0}^{d}(1-q^{\ell} t)}.
  \end{equation*}
  The statement readily follows.
\end{proof}

Thee polynomials $\qbinom{a ; x}{b}$ also have a very simple limit at $q=0$, once normalized appropriately.

\begin{lemma}
  \label{limit_q0}
  For $0 \leq i \leq d$, the polynomial
  $\qbinom{d-i ; x}{d} q^{\binom{i}{2}}$ has a limit when
  $q=0$ and this limit is $1$ if $i=0$ and $x(x-1)^i$ otherwise.
\end{lemma}
\begin{proof}
  One uses the explicit product Formula \eqref{def:qbi} for these polynomials. The statement is clear if
  $i=0$. When $i>0$, this is a simple computation term-by-term in the
  product, by distributing the factor $q^{\binom{i}{2}}$ in the
  appropriate factors to compensate for poles in $q$.
\end{proof}







\section{$q$-analogue of the incidence algebra}
\label{app:q-incidence}

The aim of this section is to explain an interpretation of the values
of $q$-Zeta polynomials as corner coefficients of powers of
matrix-like objects, similar to the classical case taking place in the
incidence algebra.

The idea of a possible $q$-analogue of the incidence algebra has been
alluded to in the last paragraph of \cite{stanley_local}.

Let us consider a poset $P$ endowed with the height function $h$.

Let $\Inc_q(P,h)$ be the vector space of square matrices $A$ with
rows and columns indexed by $P$, with coefficients in the ring of
Laurent polynomials in $q$, and such that $A_{x,y} = 0$ if $x$ is not
smaller than or equal to $y$ in $P$. Let $D_h$ be the diagonal matrix with
coefficients $q^{h(x)}$ for $x \in P$.

The space $\Inc_q(P,h)$ is endowed with the following bilinear product
\begin{equation}
  (A \times_q B)_{x,z} = A D_h B.
\end{equation}

The product $\times_q$ is associative, with the diagonal matrix
$D_h^{-1}$ as unit.

This ring is a twisted version of the usual incidence
algebra, to which it reduces when $q=1$. It can be identified with a
sub-ring of the usual ring of square matrices by multiplying all
matrices on the left by $D_h$.

\medskip

Let $Z \in \Inc_q(P,h)$ be the zeta matrix of $P$, defined by
$Z_{x,y}=1$ for all relations $x \leq y$ in $P$ and $0$ otherwise. The
matrix $Z$ is invertible in $\Inc_q(P,h)$, because it is upper triangular with
invertible diagonal.

\medskip

Let now $P$ be a bounded poset. Let $H$ be the maximal value of $h$ on
$P$. For $n \in \ZZ$, let $Z^{\times_q n}$ be the $n$-th power of $Z$
in $\Inc_q(P,h)$.

We will use the ring $A_q$ and the operator $\Delta_q$ as defined in
\S \ref{app:q-ring}.

\begin{proposition}7
  The sequence of matrices $(Z^{\times_q n})_{n \in \ZZ}$ is
  annihilated by $\Delta_q^{H+1}$.  For all $n \in \ZZ$, the value of
  the $q$-Zeta polynomial $\qZ_{P,h}([n]_q)$ is the coefficient of
  index $(\bott,\topp)$ in the matrix $Z^{\times_q n}$.
\end{proposition}
\begin{proof}
  For $n \geq 2$, the definition of the matrix $Z^{\times_q n}$ as a product
  $ZD_h ZD_h\cdots Z$ implies directly that its corner coefficient is
  the weighted sum over all chains
  $\bott \leq e_1 \leq \cdots \leq e_{n-1} \leq \topp$, where the
  weight is $q^{\sum_j h(e_j)}$. This is exactly
  \eqref{qZ_sum_over_chains}.  This implies that the sequence
  $(Z^{\times_q n}_{\bott,\topp})_{n \geq 2}$ is annihilated by $\Delta_q^{H+1}$.

  The same proof applies, for every relation $x \leq y$, to the
  sequence of coefficients $(Z^{\times_q n}_{x,y})_{n \geq 2}$, which
  is therefore also annihilated by $\Delta_q^{H+1}$. The sequence of
  matrices $(Z^{\times_q n})_{n \geq 2}$ is annihilated by
  $\Delta_q^{H+1}$ acting by \eqref{delta_on_seq}.

  Now consider the sequence $\Delta^{H+1} (Z^{\times_q n})_{n \geq N}$
  for some $N \in \ZZ$. This is
  \begin{equation}
    \Delta^{H+1} (Z^{\times_q n})_{n\geq N} = Z^{\times_q (N - 2)} \times_q \Delta^{H+1} (Z^{\times_q n} )_{n \geq 2} = 0. 
  \end{equation}

  It follows that the whole sequence $(Z^{\times_q n})_{n \in \ZZ}$ is
  annihilated by $\Delta_q^{H+1}$. Hence every sequence of
  coefficients of fixed index $(x,y)$ is a polynomial of degree at
  most $H$ evaluated at $[n]_q$, and in particular the corner
  coefficient coincides with the $q$-Zeta polynomial of $P$.
\end{proof}

Note that this statement also implies \Cref{valeur_moins1}.

\bibliographystyle{plain}
\bibliography{q-zeta-polynomial.bib}

\begin{thebibliography}{10}

\bibitem{athaBW}
Christos~A. Athanasiadis, Thomas Brady, and Colum Watt.
\newblock Shellability of noncrossing partition lattices.
\newblock {\em Proc. Amer. Math. Soc.}, 135(4):939--949, 2007.

\bibitem{posets_BHZ}
Pierre Baumann, Fr\'{e}d\'{e}ric Chapoton, Christophe Hohlweg, and Hugh Thomas.
\newblock Chains in shard intersection lattices and parabolic support posets.
\newblock {\em J. Comb.}, 9(2):309--325, 2018.

\bibitem{BHZ}
Nantel Bergeron, Christophe Hohlweg, and Mike Zabrocki.
\newblock Posets related to the connectivity set of {C}oxeter groups.
\newblock {\em J. Algebra}, 303(2):831--846, 2006.

\bibitem{bessis}
David Bessis.
\newblock The dual braid monoid.
\newblock {\em Ann. Sci. \'{E}cole Norm. Sup. (4)}, 36(5):647--683, 2003.

\bibitem{billera_flag}
Louis~J. Billera and G\'{a}bor Hetyei.
\newblock Linear inequalities for flags in graded partially ordered sets.
\newblock {\em J. Combin. Theory Ser. A}, 89(1):77--104, 2000.

\bibitem{bjorner-shellable}
Anders Bj\"{o}rner.
\newblock Shellable and {C}ohen-{M}acaulay partially ordered sets.
\newblock {\em Trans. Amer. Math. Soc.}, 260(1):159--183, 1980.

\bibitem{brady-watt-Kpi1}
Thomas Brady and Colum Watt.
\newblock {$K(\pi,1)$}'s for {A}rtin groups of finite type.
\newblock In {\em Proceedings of the {C}onference on {G}eometric and
  {C}ombinatorial {G}roup {T}heory, {P}art {I} ({H}aifa, 2000)}, volume~94,
  pages 225--250, 2002.

\bibitem{brady-watt}
Thomas Brady and Colum Watt.
\newblock Non-crossing partition lattices in finite real reflection groups.
\newblock {\em Trans. Amer. Math. Soc.}, 360(4):1983--2005, 2008.

\bibitem{q-ehrhart}
Fr\'{e}d\'{e}ric Chapoton.
\newblock {$q$}-analogues of {E}hrhart polynomials.
\newblock {\em Proc. Edinb. Math. Soc. (2)}, 59(2):339--358, 2016.

\bibitem{delapena1994}
Jos{\'e}~Antonio de~la Pe{\~n}a.
\newblock Coxeter transformations and the representation theory of algebras.
\newblock In {\em Finite dimensional algebras and related topics. Proceedings
  of the NATO Advanced Research Workshop on Representations of algebras and
  related topics. Ottawa, Canada, August 10-18, 1992}, pages 223--253.
  Dordrecht: Kluwer Academic Publishers, 1994.

\bibitem{sage}
The~SageMath Developers.
\newblock sagemath/sage: 10.0, 2023.

\bibitem{gessel-P}
Ira~M. Gessel.
\newblock A historical survey of {$P$}-partitions.
\newblock In {\em The mathematical legacy of {R}ichard {P}. {S}tanley}, pages
  169--188. Amer. Math. Soc., Providence, RI, 2016.

\bibitem{mccammond}
Jon McCammond.
\newblock Noncrossing partitions in surprising locations.
\newblock {\em Amer. Math. Monthly}, 113(7):598--610, 2006.

\bibitem{petersen}
T.~Kyle Petersen.
\newblock On the shard intersection order of a {C}oxeter group.
\newblock {\em SIAM J. Discrete Math.}, 27(4):1880--1912, 2013.

\bibitem{reading-fpsac}
Nathan Reading.
\newblock Noncrossing partitions and the shard intersection order.
\newblock In {\em 21st {I}nternational {C}onference on {F}ormal {P}ower
  {S}eries and {A}lgebraic {C}ombinatorics ({FPSAC} 2009)}, volume~AK of {\em
  Discrete Math. Theor. Comput. Sci. Proc.}, pages 745--756. Assoc. Discrete
  Math. Theor. Comput. Sci., Nancy, 2009.

\bibitem{reading}
Nathan Reading.
\newblock Noncrossing partitions and the shard intersection order.
\newblock {\em J. Algebraic Combin.}, 33(4):483--530, 2011.

\bibitem{rota1964}
Gian-Carlo Rota.
\newblock On the foundations of combinatorial theory. {I}: {Theory} of
  {M{\"o}bius} functions.
\newblock {\em Z. Wahrscheinlichkeitstheor. Verw. Geb.}, 2:340--368, 1964.

\bibitem{stanley1970}
Richard~P. Stanley.
\newblock A chromatic-like polynomial for ordered sets.
\newblock Proc. 2nd {Chapel} {Hill} {Conf}. {Combin}. {Math}. {Appl}., {Univ}.
  {North} {Carolina} 1970, 421-427 (1970)., 1970.

\bibitem{stanley_ordered}
Richard~P. Stanley.
\newblock {\em Ordered structures and partitions}.
\newblock American Mathematical Society, Providence, R.I., 1972.
\newblock Memoirs of the American Mathematical Society, No. 119.

\bibitem{stanley1974}
Richard~P. Stanley.
\newblock Combinatorial reciprocity theorems.
\newblock {\em Adv. Math.}, 14:194--253, 1974.

\bibitem{stanley_two_polytopes}
Richard~P. Stanley.
\newblock Two poset polytopes.
\newblock {\em Discrete Comput. Geom.}, 1:9--23, 1986.

\bibitem{stanley_local}
Richard~P. Stanley.
\newblock Subdivisions and local {{\(h\)}}-vectors.
\newblock {\em J. Am. Math. Soc.}, 5(4):805--851, 1992.

\bibitem{stanley_eulerian}
Richard~P. Stanley.
\newblock A survey of {Eulerian} posets.
\newblock In {\em Polytopes: abstract, convex and computational. Proceedings of
  the NATO Advanced Study Institute, Scarborough, Ontario, Canada, August 20 -
  September 3, 1993}, pages 301--333. Dordrecht: Kluwer Academic Publishers,
  1994.

\bibitem{stanleyEC1}
Richard~P. Stanley.
\newblock {\em Enumerative combinatorics. {V}olume 1}, volume~49 of {\em
  Cambridge Studies in Advanced Mathematics}.
\newblock Cambridge University Press, Cambridge, second edition, 2012.

\bibitem{thomas_tamari}
Hugh Thomas.
\newblock The {Tamari} lattice as it arises in quiver representations.
\newblock In {\em Associahedra, Tamari lattices and related structures. Tamari
  memorial Festschrift}, pages 281--291. Basel: Birkh{\"a}user, 2012.

\end{thebibliography}

\end{document}